\documentclass[11pt]{article}

\usepackage[english]{babel}
\usepackage{amsmath,amssymb,amsthm}
\usepackage{graphicx, bm}
\usepackage{multirow}

\newtheorem{thm}{Theorem}
\newtheorem{lem}{Lemma}
\newtheorem{rmk}{Remark}

\graphicspath{{./figuras/}}
\begin{document}

\title{An optimal aggregation type classifier}
\author{Alejandro Cholaquidis, Ricardo Fraiman, Juan Kalemkerian, Pamela Llop}

\begin{center}
\Large \bf A. Cholaquidis, Ricardo Fraiman, Juan Kalemkerian and Pamela Llop 
\end{center}

\abstract{We introduce a nonlinear aggregation type classifier for
functional data defined on a separable and complete metric space.
The new rule is built up from a collection of  $M$ arbitrary
training classifiers. If the classifiers are consistent, then  so is the aggregation rule. Moreover, asymptotically the aggregation rule behaves as well as the best of the $M$
classifiers. The results of a small si\-mu\-lation are reported both, for high dimensional
and functional data.}

\noindent\it Key words\rm :Functional data; supervised classification; non-linear aggregation.

 \section{Introduction} 
 Supervised classification is still one of the  hot topics for high dimensional and functional data due to the importance of their
 applications and the intrinsic difficulty in a general setup. In particular, there is a
 large list of linear aggregation methods developed in recent years, like boosting
 (\cite{breiman_96}, \cite{breiman_98}), random forest (\cite{breiman_2001},
 \cite{BDL_2008}, \cite{biau_2012}), among others. All these
 methods exhibit an important improvement when combining a subset of
 classifiers to produce a new one. Most of the contributions to the
 aggregation literature have been proposed for nonparametric
 regression, a problem closely related to
 classification rules,  which can be obtained just by plugging in the estimate of the
 regression function into the Bayes rule (see for instance, \cite{yang_2004} and \cite{bunea_2007}). Model
 selection (select the optimal single model from a list of models),
 convex aggregation (search for the optimal convex combination of a given 
 set of estimators), and linear aggregation (select the
 optimal linear combination of estimators) are important contributions among a
 large list.  Our approach is to combine, in a nonlinear way, several classifiers to construct an optimal one. We follow the ideas in Mojirsheibani \cite{mojir1} and \cite{mojir2} who introduced a combined classifier for the finite dimensional setup and showed strong consistency under someway hard to verify assumptions involving the Vapnik Chervonenkis dimension of the random partitions of the set of classifiers which are also non--valid in the functional setup. We extend the ideas to the functional setup and provide consistency results as well as rates of convergence under very mild assumptions.  We also show optimality properties of the aggregated rule which exhibit a good behavior in high dimensional and functional data. Very recently, Biau et al. \cite{biau_2013} introduced a new nonlinear aggregation strategy for the regression problem called COBRA, extending the ideas in Mojirsheibani \cite{mojir1} to the more general setup of nonparametric regression.  See also \cite{FM} for some related ideas 
regarding density estimation. 
  
In section \ref{setup} we introduce the new classifier in the general context of a separable and complete metric space which combines, in a nonlinear way, the decision of $M$ experts (classifiers). A more flexible rule is also considered. In Section \ref{resultados} we state our two main results regarding consistency, rates of convergence and asymptotic optimality of the classifier. Asymptotically, the new rule performs as the best of the $M$ classifiers used to build it up. Section \ref{simus} is devoted to show through some simulations the performance of the new classifier in high dimensional and functional data for moderate sample sizes. All proofs are given in the Appendix.

\section{The setup}\label{setup}

Throughout the manuscript $\mathcal{F}$ will denote a separable and complete metric space, $(X,Y)$ a
random pair taking values in $\mathcal{F} \times \{0,1\}$ and by $\mu$  the probability measure of
$X$. The elements of the training sample  $\mathcal{D}_n\hspace{-0.1cm}=\hspace{-0.1cm}\{(X_1,Y_1),\dots,(X_n,Y_n)\}$, are iid
random elements with the same distribution as the pair $(X,Y)$.  The regression function is denoted by $\eta(x)= \mathbb{E}(Y|X=x) = \mathbb P(Y=1|X=x)$, the Bayes rule by 
 $g^*(x)=\mathbb{I}_{\{\eta(x)>1/2\}}$ and the optimal Bayes risk by $L^*=\mathbb{P}\big(g^*(X)\neq Y\big)$.

In order to define our classifier, we split the sample $\mathcal{D}_n$ into two subsamples
$\mathcal{D}_k=\big\{(X_1,Y_1),\dots,(X_k,Y_k)\big\}$ and
$\mathcal{E}_l=\big\{(X_{k+1},Y_{k+1}),\dots,(X_n,Y_n)\big\}$ with $l=n-k\geq 1$. With
$\mathcal{D}_k$ we build up $M$ classifiers $g_{mk}:\mathcal{F}\rightarrow \{0,1\}$, $m=1,\dots,M$
which we place in the vector $\mathbf{g_k}(x) \doteq \big(g_{1k}(x),\dots,g_{Mk}(x)\big)$ and,
following some ideas in \cite{mojir1}, with $\mathcal{E}_l$ we construct
our aggregate classifier as follows,
\begin{equation}\label{clasificador}
g_T(x)=\mathbb{I}_{\{T_n(\mathbf{g_k}(x))>1/2\}},
\end{equation}
where
\begin{equation}\label{agregados}
T_n(\mathbf{g_k}(x))=\sum_{j=k+1}^n W_{n,j}(x)Y_{j},\quad x\in \mathcal{F},
\end{equation}
with weights $W_{n,j}(x)$ given by
\begin{equation}\label{pesos}
W_{n,j}(x)=\frac{\mathbb{I}_{\{\mathbf{g_k}(x)=\mathbf{g_k}(X_j)\}}}{\sum_{i=k+1}^n\mathbb{I}_{\{
\mathbf{g_k}(x)=\mathbf{g_k} (X_i)\}}}.
\end{equation}
Here, $0/0$ is assumed to be $0$.

For $0 \le \alpha < 1$ a more flexible version of the classifier, called $g_T(x,\alpha)$, can be
defined replacing the weights in (\ref{pesos}) by
\begin{equation}\label{pesos2}
W_{n,j}(x)=\frac{\mathbb{I}_{\{\frac{1}{M}\sum_{m=1}^M \mathbb{I}_{\{g_{mk}(x)
= g_{mk}(X_j)\}} \ge 1-\alpha\}}}{\sum_{i=k+1}^n\mathbb{I}_{\{\frac{1}{M}\sum_{m=1}^M
\mathbb{I}_{\{g_{mk}(x) = g_{mk}(X_i)\}} \ge 1-\alpha\}}}.
\end{equation}
More precisely, the more flexible version of the classifier (\ref{clasificador}) is given by
\begin{equation}\label{clasifgen}
g_T(x,\alpha)=\mathbb{I}_{\{T_n(\mathbf{g_k}(x),\alpha)>1/2\}},
\end{equation}
where $T_n(\mathbf{g_k}(x),\alpha)$ is defined as in (\ref{agregados}) but with the 
weights given by (\ref{pesos2}).
Observe that if we choose $\alpha=0$ in (\ref{pesos2}) and (\ref{clasifgen}) we obtain the weights given in (\ref{pesos})
and the classifier (\ref{clasificador}) respectively. 
\begin{rmk}
\begin{itemize}
\item[a)] The type of nonlinear aggregation used to define our classifiers turns out
to be quite natural. Indeed, we give a weight different from zero
to those $X_j$ which classify $x$ in the same group as the whole
set of classifiers $\mathbf{g_k}(X_j)$ (or $100(1 - \alpha) \%$ of
them). 
\item[b)] Since we are using the inverse functions of the classifiers $g_{mk}$, observations which
are far from $x$  for which the condition mentioned in a) is fulfilled are involved in the
definition of the classification rule. This may be very important in the case of high dimensional
data. This is illustrated in figure \ref{fig:example}.
\end{itemize}
\end{rmk}

 \begin{figure}[!ht]
 \begin{center}
 \includegraphics[width=0.49\textwidth]{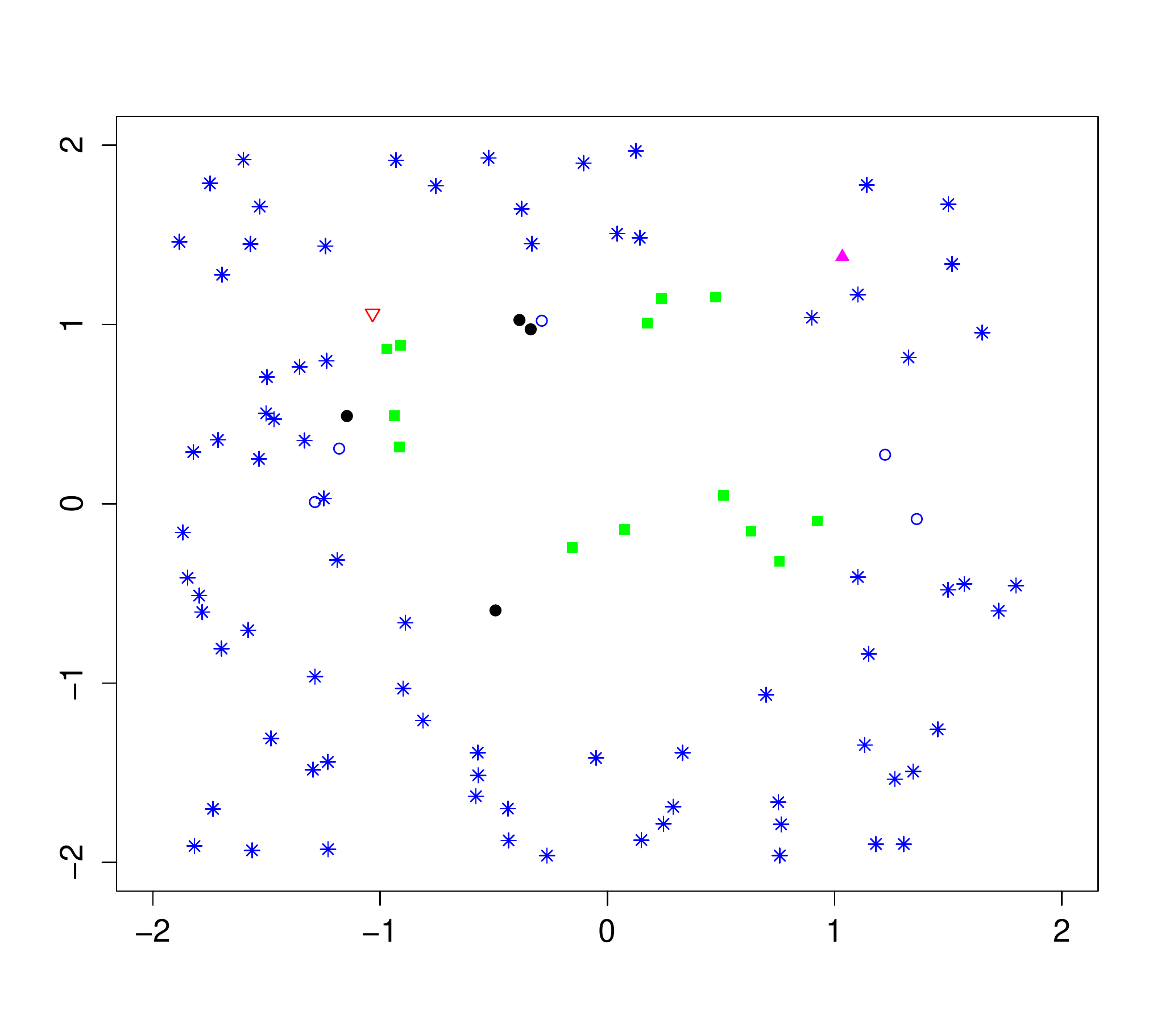}
 \includegraphics[width=0.49\textwidth]{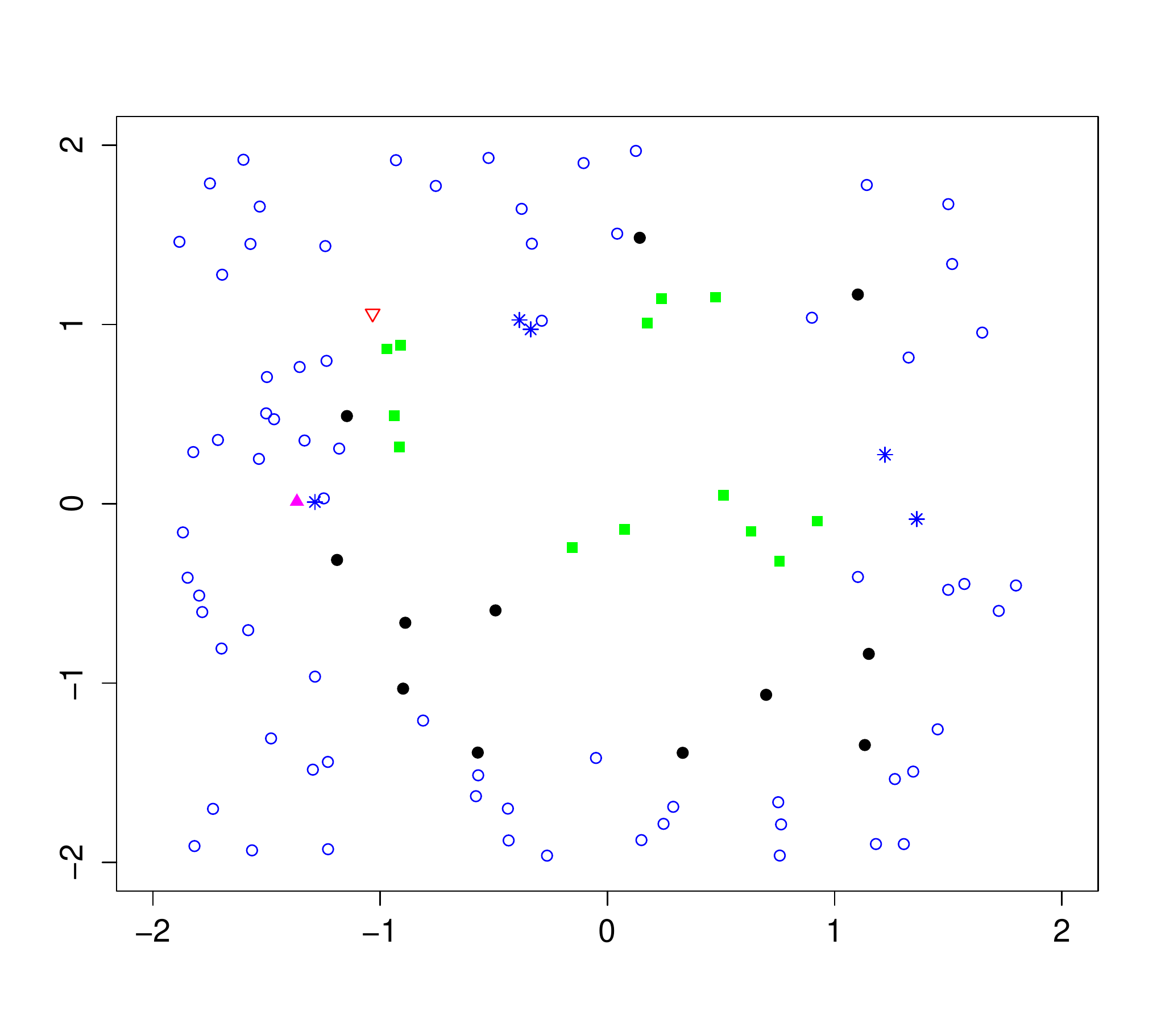}
 \end{center}
 \vspace{-1cm} 
 \caption{An example in $\mathbb R^2$ where, in both pictures, the voters of the red empty interior triangle $\nabla$  are shown as green squares while the voters of the filled magenta triangle $\Delta$  are shown with blue asterisks. The 0-class points that do not vote are shown as black dots and the 1-class points that do not vote are shown as blue circles.}
 \label{fig:example}
 \end{figure}
	
\section{Asymptotic results}\label{resultados}
In this section we show two asymptotic results for the nonlinear aggregation classifier
$g_T(X,\alpha)$ (which in particular include the corresponding result for $g_T(X)$, taking $\alpha=0$). The first one shows that the classifier
$g_T(X,\alpha)$ is consistent if, for $0\leq \alpha <0.5$, at least $R\geq (1-\alpha)M$ of them are consistent classifiers. 
Moreover, rates of convergence for $g_T(X,\alpha)$ (and $g_T(X)$) are obtained assuming we know the rates of
convergence of the $R$ consistent classifiers. The second result, shows that $g_T(X,\alpha)$ (and
in consequence $g_T(X)$) behaves asymptotically as the best of the $M$ classifiers used to build it
up. In particular, this implies that if only one of the $M$ classifiers is consistent, then our rule
is also consistent. To obtain this second result, we require a slightly stronger condition than the
one used for the first result. Throughout this section we will use the notation $\mathbb
P_{\mathcal{D}_k}(\cdot)=\mathbb
P(\cdot|\mathcal{D}_k)$.

\begin{thm} \label{consistencia} \label{cor1} Assume that for every $m=1, \ldots, R$ the classifier
$g_{mk}$
converges in probability to $g^*$ as $k \to \infty$, with $R\geq M(1-\alpha)$ and $\alpha \in [0,1/2)$. 
  Let us assume that $\mathbb P(Y=1|g^*(X)=1)>1/2$ and $\mathbb
P(Y=0|g^*(X)=0)>1/2$, then 
\begin{itemize}
\item[a)] $\lim_{min\{k,l\}\rightarrow \infty} \mathbb P_{\mathcal{D}_k}(g_T(X,\alpha)\neq
Y)-L^*=0.$
\item[b)] Let $\beta_{mk}\rightarrow 0$ as $k \to \infty$, for $m=1,\dots,R$ and $\bm{\beta_{Rk}}=\max_{m=1,\dots,R}\beta_{mk}$.
  If $ \mathbb{P}_{\mathcal{D}_k}\big(g^*(X)\neq g_{mk}(X)\big)=\mathcal{O}(\beta_{mk})$, then, for $k$ large enough,
\begin{equation*} 
 \mathbb P_{\mathcal{D}_k}(g_T(X,\alpha)\neq Y)-L^*=\mathcal{O}\Big(\max\big\{\exp(-C(n-k)),\bm{\beta_{Rk}}\big\}\Big),
\end{equation*}
for some constant $C>0$. 
\end{itemize}
\end{thm}
In order to state the optimality result we introduce some additional notation. Let $\mathbb{C}
\doteq \{0,1\}^M$. For $\nu \in \mathbb{C}$ we define the following subsets
\[
A_{\nu}^0 \doteq \bigcap_{m=1}^M g_{mk}^{-1}(\nu(m)) \cap \{Y=0\},
\hspace{0.3cm}  A_{\nu}^1 \doteq \bigcap_{m=1}^M g_{mk}^{-1}(\nu(m))
\cap \{Y=1\},
\]
\[\text{ and }  \hspace{0.1cm} A_{\nu} =A_{\nu}^0 \cup A_{\nu}^1.
\]

For each $\nu \in \mathbb{C}$, we assume that
$$
(\mathcal{H}) \hspace{2cm}  \mathbb{P}_{\mathcal{D}_k}\big((X,Y) \in A_{\nu}^1\big)
\neq \mathbb{P}_{\mathcal{D}_k}\big((X,Y) \in A_{\nu}^0\big) \hspace{0.3cm} \text{with probability
one}.
$$

\begin{thm}\label{optimalidad}
Under assumption ($\mathcal{H}$) we have,
\[
\lim_{l \to \infty} \mathbb{P}_{\mathcal{D}_k}\big(g_T(X,\alpha) \neq Y\big)  \le
\mathbb{P}_{\mathcal{D}_k}\big(g_{mk}(X) \neq Y\big),
\]
for each $m=1,\ldots, M$ which implies that, 
\[
\lim_{l \to \infty} \mathbb{P}_{\mathcal{D}_k}\big(g_T(X,\alpha) \neq Y\big)  \le \min_{1\le m \le M
}
\mathbb{P}_{\mathcal{D}_k}\big(g_{mk}(X) \neq Y\big).
\]
\end{thm}
%
%
 \section{A small simulation study}\label{simus}
 In this section we present  the performance of the aggregated classifier  in two different
 scenarios. The first one corresponds to high dimensional data, while in the second one we
consider  two simulated models for functional data analyzed in \cite{hall}. \\ 

\subsection*{High dimensional setting}
In this setting we show the performance of our method by analyzing data ge\-ne\-ra\-ted in $ \mathbb R^{150}$ in the following way: we ge\-ne\-rate $N$ iid  uniform random 
variables in $[0,1]$, say $Z_1,\ldots,Z_N$.
For each $i=1,\ldots, N$, if $Z_i >1/6$, we generate
a random  variable $X_i \in  \mathbb R^{150}$ with uniform distribution in  $[-2,2]^{150}$ and
$Y_i=1$. If $Z_i \le 1/6$, we generate a random variable $X_i\in \mathbb R^{150}$ with uniform
distribution in $\tau_v([-2,2]^{150})$ where $\tau_v$ is the translation along the direction
$(v,\dots,v)\in \mathbb{R}^{150}$ for $v = 1/4$ and set $Y_i =0$. Then we split the sample into two
 subsamples: with the first $n=N/2$ pairs $(X_i,Y_i)$, we build the training sample, with the
remaining $N/2$ we build the testing sample. 
 
We consider $M$ nearest neighbor classifiers  with the number of neighbors taken as follows:
\begin{enumerate}
\item\label{1} we fix $M=8$ consecutive odd numbers;
\item\label{2} we choose at random $M=10$ different odd integers between $1$ and \\
\mbox{$\min\{\sum_{i=1}^k Y_i, k-\sum_{i=1}^k Y_i\}$}.
\end{enumerate}
 For different sizes of $N$, $k$, $l$ and $\alpha$, we build up our
 classifier. In Table \ref{tabla1}, we report the misclassification errors for case
 \ref{1} when compared with the nearest neighbor rules build up with a sample size $n$
 taking  $5,7,9,11,13,15,17,19$ nearest neighbors (these classifiers are denoted by $g_{mk}$ for $m=1,\dots,8$).  The errors are shown in brackets  
 for a sample size $k$. In Table \ref{tabla2}, we report the misclassification errors for case
 \ref{2} and compare with the (optimal) cross validated nearest neighbor classifier.  

 \begin{table}[h]
\footnotesize{ 
 \begin{center}
  \begin{tabular}{|c|c|c|c|c|c|c|c|c|c|c|}
 \hline
 n (k)     &$g_T(\cdot)$&$g_T(\cdot,1/4)$&$g_{1k}$&$g_{2k}$&$g_{3k}$&$g_{4k}$&
$g_{5k}$&$g_{6k}$&$g_{7k}$&$g_{8k}$\\
 \hline
 
 400   &.046&.056&.071&.067&.066&.067&.068&.069&.071&.073  \\
 (300) &    &    &(.074)&(.072)&(.072)&(.073)&(.074)&(.077)&(.080)&(.082)\\   
 \hline     
 600   &.043&.052&.067&.062&.061&.061&.061&.062&.063&.065 \\
 (400)     &  &   & (.072) & (.069) & (.068) &  (.068) & (.069) & (.071) & (.073) & (.076)\\
 \hline
 800   &.037 & .045 & .062   & .057   & .055   &	.055   &	.055    & .056   & .056   &
.057\\
 (600)   & &  & (.066) & (.061) & (.060) & (.060)  & (.060) & (.061) & (.062) & (.064)\\
 \hline
 1000   &.035  & .043	   &	 .061   & .055   & .053   &.052     & .052   &.052    &	.053  &	.054
\\
 (700)     &  &  &(.065)  & (.060) & (.058) & (.057)  & (.057) & (.058) & (.059) & (.060)\\
 \hline
 \end{tabular}
 \end{center}}
 \caption{Misclassification error over $500$ for $\mathbb{R}^{150}$ with fixed number of
neighbors.}\label{tabla1}
 \end{table}

 \begin{table}[h]
 \footnotesize{
 \begin{center}
 \begin{tabular}{|c|c|c|c|c|c|c|}
 \hline
 n   & k& $g_T(\cdot)$ & $g_T(\cdot, 1/8)$ &  $g_T(\cdot, 1/4)$ & $gcv_{n}$ & $gcv_{k}$\\
 \hline
 400 &300 & .052 &.065 &	.077 & .068 &.073\\
 600 &400 & .049 &.063 &	.074 &  .061  &  .068\\
 800 &500 & .048 &.062 &	.073	 & .056 &  .061\\
 1000&700 & .047 &.061 & .072 &	.053	 & .058\\
 \hline
 \end{tabular}
 \end{center}}
 \caption{Misclassification error over $500$ repetitions for $\mathbb{R}^{150}$ with the number of neighbors chosen at random.}\label{tabla2}
 \end{table}
  
\subsection*{Functional data setting}
In this setting we show the performance of our method by analyzing the following two models considered in \cite{hall}:
\begin{itemize}
\item Model I: We generate two samples of size $n/2$ from different populations following the
model
\begin{equation*}\label{modelo-hall}	
X_{pi}(t) = \sum_{j=1}^6 \mu_{p,j} \phi_j(t) + e_{pi}(t), \hspace{1cm} p = 1,2,\hspace{0.5cm}  i =
1,\ldots, n/2,
\end{equation*}
where $\phi_j(t) = \sqrt{2} \sin(\pi j t)$, $\mu_{1,j}$ and $\mu_{2,j}$ are, respectively, the j-th coordinate of the mean vectors $\mu_{1} = (0,-0.5, 1, -0.5, 1, -0.5)$, and $\mu_{2} = (0, -0.75, 0.75, -0.15, 1.4, 0.1)$ while the errors are given by
\[
e_{pi}(t) = \sum_{j=1}^{40} \sqrt{\theta_j} Z_{pj} \phi_j(t), \hspace{1cm} p = 1,2,
\]
with $Z_{pj}\sim \mathcal{N}(0,1)$ and $\theta_j = 1/j^2$.

\item Model II: We generate two samples of size $n/2$ from different populations following the model
\begin{equation*}\label{modelo-hall2}	
X_{pi}(t) = \sum_{j=1}^3 \mu_{p,j} \phi_j(t) + e_{pi}(t), \hspace{1cm} p = 1,2,\hspace{0.5cm}  i =
1,\ldots, n/2,
\end{equation*}
where $\mu_{1} = 0.75\cdot(1, -1, 1)$ and $\mu_{2,j}$ the j-th coordinate of $\mu_{2} \equiv 0$,
$\theta_j = 1/j^2$ and the errors are given by
\[
e_{pi}(t) = \sum_{j=1}^{40} \sqrt{\theta_j} Z_{pj} \phi_j(t), \hspace{1cm} p = 1,2,
\]
with $Z_{pj}\sim \mathcal{N}(0,1)$ and $\theta_j = \exp\{- (2.1 - (j-1)/20)^2 \}$. 

This second model looks more challenging since although the means of the two populations are quite
different, the error process is very wiggly, concentrated in high frequencies (as shown in Figure
\ref{figmodiimean} left and right panel, respectively). So in this case, in order to  apply our
classification method, we have first performed the Nadaraya-Watson kernel smoother
(taking a normal kernel) to the training sample with different values of the bandwidths 
for each of the two populations. The values for the bandwidths were chosen via cross-validation 
with our classifier, varying the bandwidths between $.1$ and $.7$ (in intervals of length $.05$).
The optimal values, over 200 repetitions, were $h_1=.15$ for the first population
(with mean $\mu_1$) and $h_2=.7$ for the second one. Finally, we apply the classification
method to the raw (non-smoothed) curves of the testing sample.
\end{itemize}

\begin{figure}
\begin{center}
\includegraphics[width=0.45\textwidth]{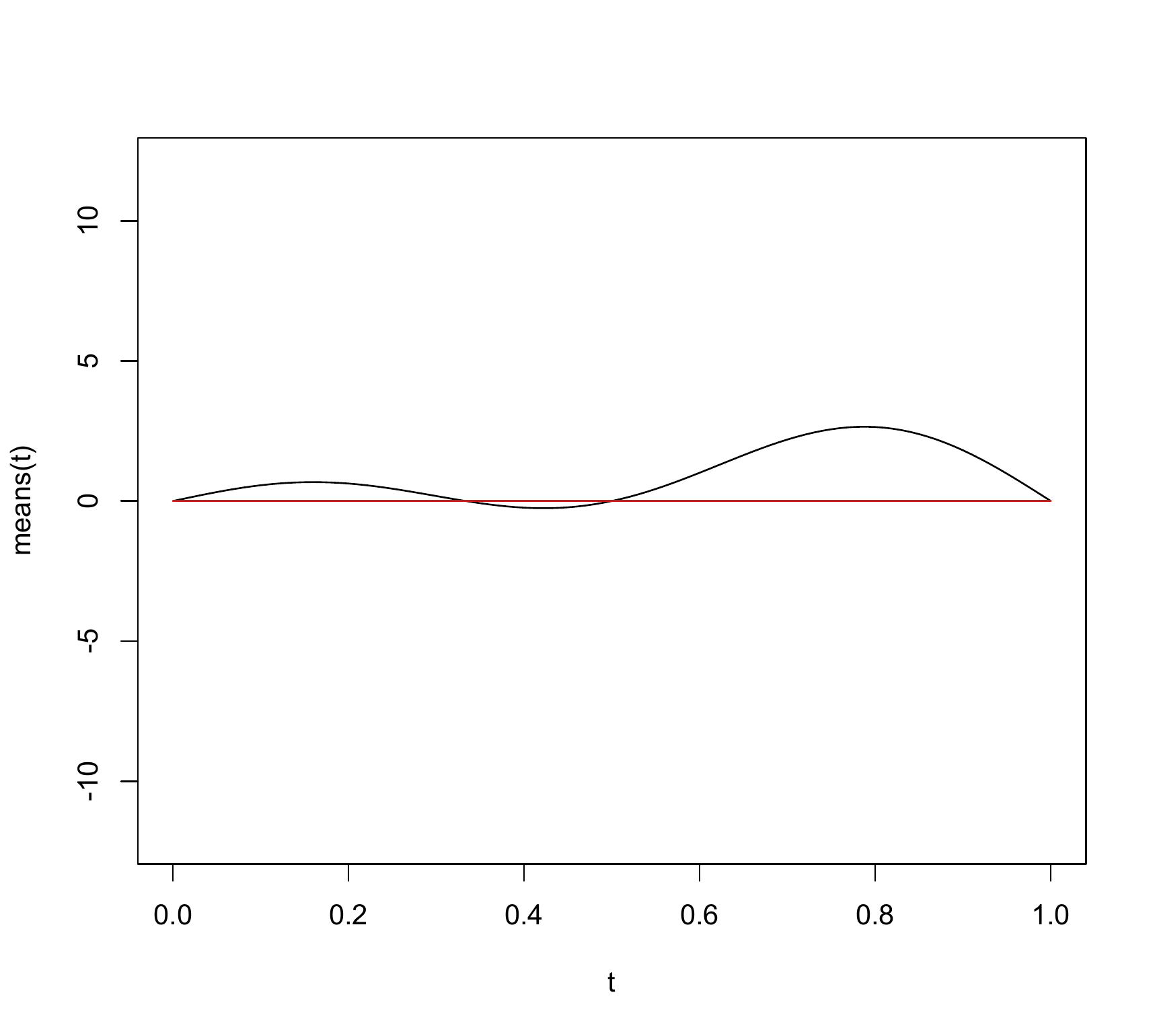} 
\includegraphics[width=0.45\textwidth]{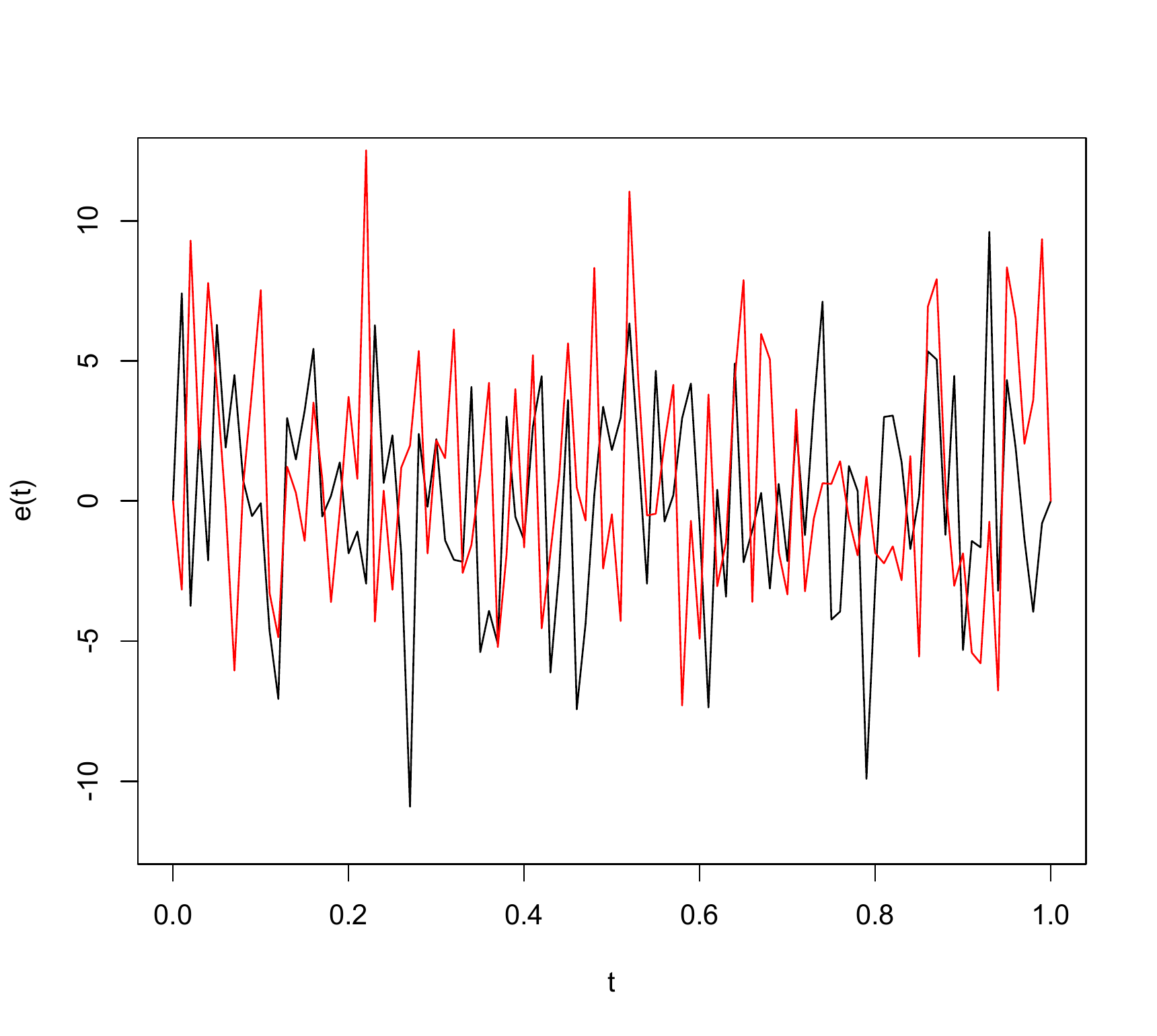}  
\caption{Mean curve (Left) and Error curve (Right) of the two populations of Model II.}
\label{figmodiimean}
\end{center}
\end{figure}
In Table \ref{tabla5} we report the misclassification error over $200$ replications for models  I and II, taking different values for $n$, $k$, $l$ and $\alpha$. In the whole training sample (of $n$
functions) the $n/2$ labels for every population were chosen at random. The test sample consist of $200$ data, taking $100$ of every population. Here, $g_{mk} = (2m-1)$-nearest neighbor rule for $m=1,\dots,5$. 
\begin{table}[!ht]
\begin{center}
\footnotesize{
\begin{tabular}{|c|c|c|c|c|c|c|c|c|c|}
\hline
 Model                  & n (k)    & $g_T(\cdot)$ & $g_T(\cdot, 1/5)$ &  $g_T(\cdot, 2/5)$ & 
$g_{1k}$ & $g_{2k}$ & $g_{3k}$  & $g_{4k}$  & $g_{5k}$\\
\hline
\multirow{4}{*}{I}       & 30       & .005         & .010              &  .014              &.032   
   &  .015    & .009      &.008       &.006\\			         			
		&(20)      &              &                   &                    & .043       &
.024     & .019      &.018      & .020\\
 \cline{2-10}
                         & 50        &.003         &  .006             &.008                & .025  
   &     .010  & .006      &.005       &.004   \\
                         &(30)       &              &                   &                    & .033 
   &     .016  & .012      &.010       &.009\\
 \hline
  \multirow{4}{*}{II}    & 30        &  .060        & .070              &  .079              &  .098
    & .074     &.073       & .074      &.077  \\ 				                    
&(20)       &              &                   &                    &  .097     &  .081     & .083  
  &.088       &.094 \\
 \cline{2-10}
                         & 50        &  .058        & .067              &  .071              & .105 
    & .074      &.067       &.065      &.067 \\
                         &(30)      &              &                   &                    &   .097
    & .075      &.073       &.076      &.080\\                          \hline 
 \end{tabular}}
 \end{center}
 \caption{Misclassification error over $200$ repetitions for models I and II.}\label{tabla5}
 \end{table}

For Model I we get a better performance than the PLS-Centroid Classifier proposed by \cite{hall}. For model II PLS-Centroid Classifier clearly outperforms our classifier although we get a quite small missclassification error, just using a combination of five nearest neighbor estimates.

\section{Concluding remarks}
 We introduce a new nonlinear aggregating method for supervised
classification in a general setup built up from a family of
 classifiers $g_{1k},\dots,g_{Mk}$. We prove consistency, rates of convergence and a certain  kind
of optimality, in the sense that the nonlinear aggregation rule behaves asymptotically as  well as
the best one among  the $M$ classifiers $g_{1k},\dots,g_{Mk}$. A small simulation study  confirms
the asymptotic results for moderate sample sizes. In particular it is well behaved for 
high--dimensional and functional data.
 
\section{Appendix: Proof of results}
To prove Theorem \ref{consistencia} we will need the following Lemma.
\begin{lem}\label{lema}
Let $f(x)$ be a classifier built up from the training sample $\mathcal{D}_k$ such that
$\mathbb P_{\mathcal{D}_k}(f(X)\neq g^*(X)) \rightarrow 0$ when
$k\rightarrow \infty$. Then, $\mathbb P_{\mathcal{D}_k}(f(X)\neq Y)-L^*\rightarrow 0$.
\end{lem}
\begin{proof} [Proof of Lemma \ref{lema}]
First we write,
\begin{align}\label{error}
\mathbb P_{\mathcal{D}_k}\big(f(X)\neq Y\big)- L^* &= \mathbb P_{\mathcal{D}_k}\big(f(X)\neq
Y\big)-P\big(g^*(X)\neq Y\big) \nonumber\\ &= \mathbb P_{\mathcal{D}_k}\big(f(X)\neq Y,Y=g^*(X)\big)
\nonumber\\ &\hspace{0.3cm}+ \mathbb P_{\mathcal{D}_k}\big(f(X)\neq Y, Y\neq g^*(X)\big)
-P\big(g^*(X)\neq Y\big)\nonumber \\ &= \mathbb P_{\mathcal{D}_k}\big(f(X)\neq g^*(X)\big)\\
&\hspace{0.3cm} + \mathbb P_{\mathcal{D}_k}\big(f(X)\neq Y, Y\neq g^*(X)\big)-P\big(g^*(X)\neq
Y\big) \nonumber \\ &= \mathbb P_{\mathcal{D}_k}\big(f(X)\neq g^*(X)\big) -
\mathbb P_{\mathcal{D}_k}\big(g^*(X)\neq Y, f(X)= Y\big), \nonumber
\end{align}
where in the last equality we have used that 
\[
P\big(g^*(X)\neq Y\big)= \mathbb P_{\mathcal{D}_k}\big(g^*(X)\neq Y, f(X)\neq Y\big) +
\mathbb P_{\mathcal{D}_k}\big(g^*(X)\neq Y, f(X)= Y\big),
\]
implies 
\[
 \mathbb P_{\mathcal{D}_k}\big(g^*(X)\neq Y,f(X)= Y\big) = \mathbb P_{\mathcal{D}_k}\big(g^*(X)\neq
Y, f(X)\neq Y\big) - P\big(g^*(X)\neq Y\big).
\]
Therefore, replacing in (\ref{error}) we get that
\begin{align} \label{boundl3}
\mathbb P_{\mathcal{D}_k}\big(f(X)\neq Y\big)- L^* &=  \mathbb P_{\mathcal{D}_k}\big(f(X)\neq
g^*(X)\big) -
\mathbb P_{\mathcal{D}_k}\big(g^*(X)\neq Y, f(X)= Y\big)\nonumber \\ &\le \mathbb
P_{\mathcal{D}_k}\big(f(X)\neq g^*(X)\big),
\end{align}
which by hypothesis converges to zero as $k \to \infty$ and the Lemma is proved.
\end{proof}
\begin{proof} [Proof of Theorem \ref{consistencia}]
We will prove part b) of the Theorem since part a) is a direct consequence of it. By (\ref{boundl3}), 
it suffices to prove that, for $k$ large enough: 
$$\mathbb P_{\mathcal{D}_k}(g_T(X,\alpha)\neq g^*(X))=\mathcal{O}\Big(\max\big\{\exp(-C(n-k)),\bm{\beta_{Rk}}\big\}\Big).$$ 
We first split $\mathbb P_{\mathcal{D}_k}(g_T(X,\alpha)\neq g^*(X))$ into two terms,
\begin{align*} 
\mathbb P_{\mathcal{D}_k}(g_T(X,\alpha)\neq g^*(X)) &= \mathbb P_{\mathcal{D}_k}(g_T(X,\alpha)\neq g^*(X),g^*(X)=1) \nonumber \\ &\hspace{0.3cm}+\mathbb P_{\mathcal{D}_k}(g_T(X,\alpha)\neq g^*(X),g^*(X)=0) \doteq I +II.
\end{align*}
Then we will prove that, for $k$ large enough,
$I=\mathcal{O}\Big(\max\big\{\exp(-C_1(n-k)),\bm{\beta_{Rk}}\big\}\Big)$. The proof that 
$II=\mathcal{O}\Big(\max\big\{\exp(-C_2(n-k)),\bm{\beta_{Rk}}\big\}\Big)$ is completely analogous and we omit it.
Finally, taking $C=\min\{C_1,C_2\}$, the proof will be completed.
In order to deal with term $I$, let us define the vectors 
\begin{align*}
\mathbf{g_{Rk}}(X) =\  & \big(g_{1k}(X),\dots,g_{Rk}(X)\big)\in \{0,1\}^R,\\ 
\nu(X)= \ & \Big(1,\dots,1,g_{(R+1)k}(X),\dots,g_{Mk}(X)\Big)\in \{0,1\}^M.
\end{align*}
Then, 
\begin{align*}
I &= \mathbb P_{\mathcal{D}_k}(g_T(X,\alpha)\neq g^*(X),g^*(X)=1)  \nonumber\\
 &\le \mathbb P_{\mathcal{D}_k}(g_T(X,\alpha)\neq g^*(X),g^*(X)=1, \mathbf{g_{Rk}}(X)=\mathbf{1})
\nonumber \\ &\hspace{0.5cm}+\sum_{m=1}^R \mathbb P_{\mathcal{D}_k}(g_T(X,\alpha)\neq
g^*(X),g^*(X)=1,g_{mk}(X)=0)\nonumber \\
&\le \mathbb P_{\mathcal{D}_k}(g_T(X,\alpha)\neq g^*(X),g^*(X)=1,
\mathbf{g_{Rk}}(X)=\mathbf{1})\nonumber \\ 
&\hspace{0.5cm}+\sum_{m=1}^R \mathbb P_{\mathcal{D}_k}(g^*(X)\neq g_{mk}(X)) \\ &\le \mathbb
P_{\mathcal{D}_k}\Big(T_n(\mathbf{g_k}(X), \alpha)\leq
1/2\big|g^*(X)=1,\mathbf{g_{Rk}}(X)=\mathbf{1}\Big) 
 \\ &\hspace{0.5cm}+\sum_{m=1}^R \mathbb P_{\mathcal{D}_k}(g^*(X)\neq g_{mk}(X))
\\ &\doteq I_A + I_B.\nonumber 
\end{align*} 
Observe that, conditioning to $\mathbf{g_{Rk}}(X)=\mathbf{1}$ and defining 
$$
Z_{j} \doteq \mathbb{I}_{\left\{\frac{1}{M}\sum_{m=1}^M \mathbb{I}_{\{g_{mk}(X_j)= \nu(m)\}} \ge
1-\alpha\right\}},
$$
where $\nu(m)=\nu(X)(m)$ is the $m$-th entry of the vector $\nu(X)$, we can rewrite
$T_n(\mathbf{g_{k}}(X), \alpha) $ as
\[
T_n(\mathbf{g_{k}}(X), \alpha) = \frac{\sum_{j=k+1}^n  Z_j Y_j}{\sum_{i=k+1}^n Z_i}.
\]
Therefore, 
\begin{align}\label{proba}
I_A&= \mathbb P_{\mathcal{D}_k}\left(\frac{\frac{1}{n-k}\sum_{j=k+1}^n  Z_j
Y_j}{\frac{1}{n-k}\sum_{i=k+1}^n Z_i} \le \frac{1}{2} \Big| g^*(X)=1,
\mathbf{g_{Rk}}(X)=\mathbf{1}\right)\nonumber \\ &= \mathbb
P_{\mathcal{D}_k}\left(\frac{1}{n-k}\sum_{j=k+1}^n  Z_j (Y_j-1/2) \le 0  \Big| g^*(X)=1,
\mathbf{g_{Rk}}(X)=\mathbf{1}\right). 
\end{align}
In order to use a concentration inequality to bound this probability, we need to compute the
expectation of $Z_j (Y_j-1/2) =
Z_j Y_j - Z_j/2 $. To do this, observe that 
\[
E(Z_j Y_j)=\mathbb P_{\mathcal{D}_k}\left(\frac{1}{M}\sum_{m=1}^M \mathbb{I}_{\{g_{mk}(U)= \nu(m) \}} \ge
1-\alpha,V=1\right),
\]
and 
\begin{equation*}\label{esp1}
E(Z_j) = \mathbb P_{\mathcal{D}_k}\left(\frac{1}{M}\sum_{m=1}^M \mathbb{I}_{\{g_{mk}(U)=\nu(m)\}} \ge
1-\alpha \right),
\end{equation*}
being $(U,V)$ independent of $(X,Y)$, $\mathcal{D}_k$ and with the same law as $(X,Y)$. Since
\begin{equation*} \label{inc1}
\left\{ \mathbf{g_{Rk}}(U)=\mathbf{1} \right\} \subset \left\{ \frac{1}{M}\sum_{m=1}^M
\mathbb{I}_{\{g_{mk}(U)= \nu(m)\}} \ge 1-\alpha \right\}\doteq A_\alpha,
\end{equation*}
we have that
\begin{align}\label{esp2}
E(Z_j Y_j)-E(Z_j)/2 & = \mathbb P_{\mathcal{D}_k}(V=1|A_\alpha)\mathbb P_{\mathcal{D}_k}(A_\alpha)-\mathbb P_{\mathcal{D}_k}(A_\alpha)/2 \nonumber\\
& = \mathbb P_{\mathcal{D}_k}(A_\alpha)\Big[\mathbb P_{\mathcal{D}_k}(V=1|A_\alpha)-1/2\Big]\\
& \geq \mathbb{P}_{\mathcal{D}_k}\big(\mathbf{g_{Rk}}(U)=1\big)\Big[\mathbb
P_{\mathcal{D}_k}(V=1|A_\alpha)-1/2\Big].\nonumber
\end{align}
Now, since for $m=1,\ldots,R$, $g_{mk} \to g^*$ in probability as $k \to \infty$, 
\begin{equation} \label{lim1}
\mathbb P_{\mathcal{D}_k}\left( \mathbf{g_{Rk}}(U)=\mathbf{1}\right) \to \mathbb{P}(g^*(U)=1)\doteq p^*>0.
\end{equation}
On the another hand, we have that, for $k$ large enough,
$\mathbb P_{\mathcal{D}_k}(V=1|A_\alpha)> 1/2$. Indeed, for $m=1,\dots,R$, let us consider the
events $B_{mk}=\{g_{mk}(U)=g^*(U)\}$ which, by hypothesis, for $k$ large enough verify
$$
\mathbb P \big(\cap_{m=1}^R B_{mk}\big)>1-\varepsilon,
$$
for all $\varepsilon>0$. In particular, we can take $\varepsilon>0$ such that $\mathbb
P(V=1|g^*(U)=1)(1-\varepsilon)>1/2$. This implies that 
\begin{align} \label{eq}
\mathbb P_{\mathcal{D}_k}(V=1|A_\alpha)= \ & \frac{\mathbb P_{\mathcal{D}_k}(V=1,A_\alpha,\cap_{m=1}^R B_{mk})}{\mathbb P_{\mathcal{D}_k}(A_\alpha)} \nonumber\\
&\quad + \frac{\mathbb P_{\mathcal{D}_k}(V=1,A_\alpha,(\cap_{m=1}^R B_{mk})^c)}{\mathbb P_{\mathcal{D}_k}(A_\alpha)}\nonumber\\
\geq \ &\frac{\mathbb P_{\mathcal{D}_k}(V=1,A_\alpha,\cap_{m=1}^R B_{mk})}{\mathbb
P_{\mathcal{D}_k}(A_\alpha)}\\
= \ &\frac{\mathbb P_{\mathcal{D}_k}(V=1,A_\alpha\big|\cap_{m=1}^RB_{mk})} {\mathbb
P_{\mathcal{D}_k}(A_\alpha)}(1-\varepsilon).\nonumber
\end{align}
Conditioning to $\cap_{m=1}^R B_{mk}$ the event $A_\alpha$ equals $C_\alpha$ given by
\begin{equation} \label{condaalpha}
\left\{R\mathbb{I}_{\{g^*(U)= 1\}} +\sum_{m=R+1}^M \mathbb{I}_{\{g_{mk}(U)= \nu(m)\}}\ge M(1-\alpha)
\right\}\doteq C_\alpha.
\end{equation}
However, $\alpha<1/2$ imply that $C_\alpha=\{g^*(U)=1\}$. Indeed, from the inequality  $R\geq M(1-\alpha)$, 
it is clear that $\{g^*(U)=1\}\subset C_\alpha$. 
On the other hand, $R\geq M(1-\alpha)>M/2$ and $\alpha<1/2$ imply that $M-R<M/2<M(1-\alpha)$, 
and so the sum in the second term of (\ref{condaalpha}) is at most $M-R$ and consequently, 
$\{g^*(U)=1\}^c\subset C_\alpha^c$. Then, combining this fact with (\ref{eq}) we have that, for $k$ large enough

\begin{align}\label{otraeq}
\mathbb P_{\mathcal{D}_k}(V=1|A_\alpha)&\geq \  \frac{\mathbb
P_{\mathcal{D}_k}(V=1,g^*(U)=1\big|\cap_{m=1}^RB_{mk})}{\mathbb
P_{\mathcal{D}_k}(g^*(U)=1)}(1-\varepsilon) \nonumber \\
&=\ \frac{\mathbb P_{\mathcal{D}_k}(V=1,g^*(U)=1)}{\mathbb P_{\mathcal{D}_k}(g^*(U)=1)}(1-\varepsilon)\\
&=\ \mathbb P(V=1|g^*(U)=1)(1-\varepsilon) \nonumber \\& >1/2. \nonumber
\end{align}
Therefore, from (\ref{lim1}) and (\ref{otraeq}) in (\ref{esp2}) we get 
\[
 E(Z_j Y_j)-E(Z_j)/2  > c >0.
\]
Going back to (\ref{proba}), conditioning to $\nu(X)$ and using the Hoeffding inequality 
for $|Z_j (Y_j-1/2)| \le 1/2$, for $k$ large enough we have
\begin{small}
\begin{align*}
I_A 
&= \mathbb P_{\mathcal{D}_k}\left(\frac{1}{n-k}\sum_{j=k+1}^n  -\big(Z_j (Y_j-1/2) - E(Z_j
(Y_j-1/2)) \big)\ge c \Big| g^*(X)=1, \mathbf{g_{Rk}}(X)=\mathbf{1}\right) \\ &\le 
\exp \left\{- C_1(n-k)\right\},
\end{align*}
\end{small}
with $C_1 = 2c^2$. On the another hand, by hypothesis we have  
\begin{equation*}
I_B=\sum_{m=1}^M \mathbb P_{\mathcal{D}_k}(g^*(X)\neq g_{mk}(X))=\mathcal{O}(\bm{\beta_{Rk}}),
\end{equation*}
which concludes the proof. 
\end{proof}

\begin{proof}[Proof of Theorem \ref{optimalidad}]
First we write,
\begin{align*}
\mathbb{P}_{\mathcal{D}_k}\big(g_T(X,\alpha) \neq Y \big) &=
\mathbb{P}_{\mathcal{D}_k}\big(T_n(\mathbf{g_k}(X),\alpha) > 1/2, Y = 0\big) \\ &\hspace{1cm}+
\mathbb{P}_{\mathcal{D}_k}\big(T_n(\mathbf{g_k}(X),\alpha) \leq 1/2, Y = 1\big) \\ &\doteq
I + II.
\end{align*}
Then,
 \begin{align*}
I &= \sum_{\nu \in \mathbb{C}}\mathbb{P}_{\mathcal{D}_k}\big(T_n(\mathbf{g_k}(X),\alpha) > 1/2,
(X,Y) \in A_{\nu}^0\big) \\ &= \sum_{\nu \in \mathbb{C}}
\mathbb{P}_{\mathcal{D}_k}\left(\frac{\sharp\{j:(X_j,Y_j) \in A_{\nu}^1\}}{l}
> \frac{\sharp\{j:(X_j,Y_j) \in A_{\nu}^0\}}{l}, (X,Y) \in
A_{\nu}^0 \right),
\end{align*}
and,
 \begin{align*}
 II &= \sum_{\nu \in \mathbb{C}}
\mathbb{P}_{\mathcal{D}_k}\big(T_n(\mathbf{g_k}(X),\alpha) \le 1/2, (X,Y) \in
A_{\nu}^1\big) \\ &= \sum_{\nu \in \mathbb{C}}
\mathbb{P}_{\mathcal{D}_k}\left(\frac{\sharp\{j:(X_j,Y_j) \in A_{\nu}^1\}}{l} \leq
\frac{\sharp\{j:(X_j,Y_j) \in A_{\nu}^0\}}{l}, (X,Y) \in A_{\nu}^1
\right).
\end{align*}
Therefore, since for all $\nu$, $ \mathbb{P}_{\mathcal{D}_k}\big((X,Y) \in
A_{\nu}^1\big) \neq \mathbb{P}_{\mathcal{D}_k}\big((X,Y) \in A_{\nu}^0\big)$ almost surely,
\begin{align}\label{combinado}
\lim_{l \to \infty} \mathbb{P}_{\mathcal{D}_k}\big(g_T(X,\alpha) \neq Y \big)  &=
\sum_{\nu\in \mathbb{C}} \mathbb{I}_{\big\{\mathbb{P}_{\mathcal{D}_k}\big((X,Y) \in A_{\nu}^1\big) >
\mathbb{P}_{\mathcal{D}_k}\big((X,Y) \in A_{\nu}^0\big)\big\}} \mathbb{P}_{\mathcal{D}_k}\big((X,Y) \in A_{\nu}^0\big)\big) \nonumber \\
&+ \sum_{\nu\in \mathbb{C}} \mathbb{I}_{\big\{\mathbb{P}_{\mathcal{D}_k}\big((X,Y) \in A_{\nu}^1\big) <
\mathbb{P}_{\mathcal{D}_k}\big((X,Y) \in A_{\nu}^0\big)\big\}} \mathbb{P}_{\mathcal{D}_k}\big((X,Y) \in
A_{\nu}^1\big)\big) \nonumber \\ &= \sum_{\nu\in \mathbb{C}} \min\Big\{
\mathbb{P}_{\mathcal{D}_k}\big((X,Y) \in
A_{\nu}^1\big), \mathbb{P}_{\mathcal{D}_k}\big((X,Y) \in A_{\nu}^0\big)\Big\}.
\end{align}
On the other hand,
\begin{align}\label{unosolo}
\noindent \mathbb{P}_{\mathcal{D}_k}\big((g_{mk}(X) \neq Y\big) &=
\mathbb{P}_{\mathcal{D}_k}\big(g_{mk}(X)=0, Y=1\big) + \mathbb{P}_{\mathcal{D}_k}\big(g_{mk}(X)=1,
Y=0\big) \nonumber \\ &=\mathbb{P}_{\mathcal{D}_k} \left( \bigcup_{\nu : \nu(i) = 0} (X,Y) \in
A_{\nu}^1 \right) + \mathbb{P}_{\mathcal{D}_k}
\left(\bigcup_{\nu : \nu(i) = 1} (X,Y) \in A_{\nu}^0\right)  \nonumber
\\ &= \sum_{\nu : \nu(i) = 0} \mathbb{P}_{\mathcal{D}_k}\big((X,Y) \in A_{\nu}^1\big) + \sum_{\nu :
\nu(i) = 1} \mathbb{P}_{\mathcal{D}_k}\big((X,Y) \in A_{\nu}^0\big)\nonumber \\ &\ge \sum_{\nu\in
\mathbb{C}} \min\Big\{\mathbb{P}_{\mathcal{D}_k}\big((X,Y) \in
A_{\nu}^1\big),\mathbb{P}_{\mathcal{D}_k}\big((X,Y) \in A_{\nu}^0\big)\Big\}.
\end{align}
Therefore, from (\ref{combinado}) and (\ref{unosolo}), for each $m$ we get
\[
\lim_{l \to \infty} \mathbb{P}_{\mathcal{D}_k}\big(g_T(X,\alpha) \neq Y \big)  \le
\mathbb{P}_{\mathcal{D}_k}\big(g_{mk}(X) \neq Y\big).
\]
\end{proof}

\section*{Acknowledgment} We would like to thank Gerard Biau and James Malley for helpful suggestions.

\end{document}